\documentclass[]{article}
\usepackage{amsmath}
\usepackage{amsfonts}
\usepackage{amssymb}
\usepackage{amsthm}
\usepackage{mathrsfs}
\usepackage{enumerate}
\usepackage{tikz}
\usepackage{MnSymbol,faktor, systeme, bbm} 
\usepackage{array}
\usepackage[all]{xy}
\usepackage{subfig}
\usepackage{hyperref}

\usepackage{pgf}

\DeclareMathOperator{\Hom}{Hom}

\DeclareMathOperator{\Der}{Der}
\DeclareMathOperator{\Sym}{Sym}
\DeclareMathOperator{\Alt}{Alt}
\DeclareMathOperator{\Res}{Res}
\DeclareMathOperator{\Fix}{Fix}

\renewcommand{\divides}{\mid}

\newcommand{\integers}{\ensuremath{\mathbb{Z}}}
\newcommand{\Z}{\ensuremath{\mathbb{Z}}}
\newcommand{\rationals}{\ensuremath{\mathbb{Q}}}

\newcommand{\Gab}[2]{\ensuremath{G_{#1,#2}}}
\newcommand{\Gbar}{\ensuremath{\overline{G}}}

\newcommand{\mc}[1][n]{\ensuremath{\maxsubgr[#1]^c(\Gbar)}}

\newcommand{\maxsubgr}[1][n]{\ensuremath{m_{#1}}}
\newcommand{\subgr}[1][n]{\ensuremath{a_{#1}}}

\renewcommand{\hom}[1][n]{\ensuremath{h_{#1}}}
\newcommand{\trans}[1][n]{\ensuremath{t_{#1}}}
\newcommand{\prim}[1][n]{\ensuremath{p_{#1}}}

\newcommand{\normal}{\ensuremath{\trianglelefteq}}

\newtheorem{proposition}{Proposition}
\newtheorem{theorem}[proposition]{Theorem}
\newtheorem*{nntheorem}{Theorem}

\newtheorem{corollary}[proposition]{Corollary}
\newtheorem{lemma}[proposition]{Lemma}

\title{Subgroup growth of all Baumslag-Solitar groups}
\author{Andrew James Kelley}
\date{}

\begin{document}

\maketitle

\begin{abstract}
 This paper gives asymptotic formulas for the subgroup growth and maximal subgroup growth of all Baumslag-Solitar groups.
\end{abstract}

\section{Introduction}
For a finitely generated group $G$, let $\subgr(G)$ denote the number of subgroups of $G$ of index $n$, and let $\maxsubgr(G)$ denote
the number of maximal subgroups of $G$ of index $n$.
Also, for $a$, $b$ nonzero integers, let $\Gab{a}{b}$ denote the Baumslag-Solitar group $ \langle x, y \mid y^{-1}x^ay = x^b \rangle.$

In \cite{Gelman}, Gelman counts $\subgr(\Gab{a}{b})$ exactly for the case when $\gcd(a,b) = 1$.
Exact formulas in the area of subgroup growth are rare, and so his formula (Theorem~\ref{thm:Gelman} below) is indeed very nice.
Can a simple formula also be given for $\maxsubgr(\Gab{a}{b})$ when $\gcd(a,b) = 1$?  Yes, see Corollary~\ref{cor:maxsubgr(Gab) when gcd(a,b) = 1}.
Also, the question naturally
arises, what about the case when $\gcd(a,b) \neq 1$? 

From the work of Moldavanskii \cite{Moldavanskii}, it is apparent that the largest residually finite quotient of 
$\Gab{a}{b}$ is a group $\Gbar_{a,b}$ (which for simplicity will be denoted $\Gbar$)
which has a normal subgroup of the form $A \cong \integers[1/k]$ (for appropriate $k$) with 
$\Gbar/A \cong \integers * \integers/m\integers$, where $m = \gcd(a,b)$. When $m = 1$, this explains why the formula for 
$\subgr(\Gab{a}{b})$ is so simple; $\Gbar$ turns out to be of the form $\integers[1/k] \rtimes \integers$, and so
 Section~\ref{sec:redoing Gelman's formula} gives a more enlightening proof of Gelman's formula.
 
 When $m = \gcd(a,b) > 1$, one has to deal with the free product $\integers * \integers/m\integers$. 
 In \cite{Muller}, M\"uller studies such groups (and in fact many more: any free product of groups that are either finite or free).
 Combining this with the main result from Babai and Hayes' \cite{Babai-Hayes}, one can give an asymptotic formula for
 $\maxsubgr(\integers * \integers/m\integers)$ (and also for $\subgr(\integers * \integers/m\integers$)). Note that M\"uller's main results are even better 
 than asymptotic formulas.
 
 Next, a small argument shows that the vast majority of maximal subgroups (of any fixed, large index) of $\Gab{a}{b}$ contain the normal subgroup $A$
 (mentioned above), and hence, we obtain
 an asymptotic formula for $\maxsubgr(\Gab{a}{b})$. As it turns out, the vast majority of \emph{all} subgroups
 of $\Gab{a}{b}$ (of any fixed, large index) contain $A$, but it takes a little more work to show this. As a result, we can combine the two main results of this paper, 
 Theorems~\ref{thm:baumslag-solitar - maximal subgroup growth} and
 \ref{thm:baumslag-solitar - subgroup growth}, to obtain the following.
 \begin{nntheorem}
 	Let	$m = \gcd(a,b) > 1$. 	Then
 	\[
 	\subgr(\Gab{a}{b}) \sim \maxsubgr(\Gab{a}{b}) \sim K_m n^{(1-1/m)n + 1} \exp\left(-(1 - 1/m)n + \sum_{\substack{d < m \\ d|m}} \frac{n^{d/m}}{d} \right),
 	\]
 	where
 	\[
 	K_m :=
 	\begin{cases}
 	m^{-1/2} & \text{if $m$ is odd}\\
 	m^{-1/2}e^{-1/(2m)} & \text{otherwise.}
 	\end{cases}
 	\]
 \end{nntheorem}
This formula is based on formula (8) from \cite{Muller}.

All of the work in this paper except for Section \ref{sec:when gcd(a,b) > 1, formula for subgr(Gab)}, was completed while the author was
a graduate student at Binghamton University. Hence, most of this paper is from \cite{kelley-dissertation}, the author's dissertation.

For related work on the Baumslag-Solitar groups, note that in \cite{Button}, Button gives an exact formula for counting the normal subgroups of any index in
$\Gab{a}{b}$, when $\gcd(a,b) = 1$. For a survey of subgroup growth up until 2003, see \cite{Lubotzky-and-Segal}, the book by Lubotzky and Segal.

The notation used here is standard. The number of primitive permutation representations of $G$ of degree $n$ is denoted $\prim(G)$, and
the number of transitive permutation representations of $G$ of degree $n$ is denoted $\trans(G)$. If a group $G$ acts on a group $N$, the set
of derivations (1-cocycles) from $G$ do $N$ is denoted $\Der(G,N)$.

The goal of Section \ref{sec:largest residually finite quotient} is to describe $\Gbar$, the largest residually finite quotient of $\Gab{a}{b}$.
In Section~\ref{sec:redoing Gelman's formula}, a new proof is given for Gelman's formula, and it is shown
there what it simplifies to for $\maxsubgr(\Gab{a}{b})$. In Section~\ref{sec:when gcd(a,b) > 1, a formula for maxsubgr(Gab)}, an asymptotic
formula is given for $\maxsubgr(\Gab{a}{b})$ when $\gcd(a,b) > 1$. Finally, in Section~\ref{sec:when gcd(a,b) > 1, formula for subgr(Gab)},
it is shown that the asymptotic formula for $\maxsubgr(\Gab{a}{b})$ is also asymptotic to $\subgr(\Gab{a}{b})$ (where still $\gcd(a,b) > 1$).

\section{The largest residually finite quotient}
\label{sec:largest residually finite quotient}
The goal of this section is Corollary \ref{cor:Gbar - short exact sequence - with stated action}.

We will denote the intersection of all finite index subgroups of $\Gab{a}{b}$ by $\Res(\Gab{a}{b})$. 
In \cite{Moldavanskii}, Moldavanskii determines what $\Res(\Gab{a}{b})$ is.
Let $d := \gcd(a,b)$.

\begin{theorem}[Moldavanskii, 2010]
	\label{thm:Moldavanskii on Res(Baumslag-Solitar)}
	The group $\Res(\Gab{a}{b})$ is the normal closure in \Gab{a}{b} of the set of commutators
	$\{[y^{k}x^dy^{-k},x] : k \in \integers \}.$
\end{theorem}
Let $\Gbar = \Gab{a}{b}/\Res(\Gab{a}{b})$ 
the largest residually finite quotient of $\Gab{a}{b}$. ($\Gbar$ does depend on $a$ and $b$.)
We then have the following presentation of $\Gbar$:
\[
\Gbar = \langle x, y \mid y^{-1}x^ay = x^b, [y^{k}x^dy^{-k},x] \text{ for all } k \in \integers \rangle.
\]

We next define a subgroup of $\Gbar$ (denoted $\overline{C}$ in \cite{Moldavanskii}):
\[
A := \langle y^kx^dy^{-k} : k \in \integers \rangle \leq \Gbar.
\]
\begin{lemma}[Moldavanskii]
	\label{lem:Gab has A as an abelian normal subgroup}
	The group $A$ is an abelian normal subgroup of $\Gbar$.
\end{lemma}
Note: This is a small part of Propositions 3 and 4 in \cite{Moldavanskii}.
\begin{proof}
	We have $A \normal \Gbar$ because conjugating the generators of $A$ by $y$ just shifts them and because 
	$x$ commutes with all the generators (because of the commutators in $\Res(\Gab{a}{b})$).
	
	We have that $[y^{k}x^dy^{-k},x] \in \Res(\Gab{a}{b})$ implies that 
	$x^d$ commutes with $y^{k}x^dy^{-k}$, and hence for all $j, k \in \integers$ we get
	$[y^{k}x^dy^{-k},y^{j}x^dy^{-j}] \in \Res(\Gab{a}{b})$.
\end{proof}
It turns out that $\Gbar/A$ is the free product of a finite cyclic group with the infinite cyclic group:
(Recall that $d := \gcd(a,b)$.)
\begin{corollary}[Moldavanskii]
	\label{cor:Gbar/A is a free product}
	The group $\Gbar/A$ has presentation $\langle x, y \mid x^d \rangle$, and therefore, $\Gbar/A \cong (\integers/d\integers) * \integers$.
\end{corollary}
Note: The $x$ here does indeed correspond to the $x$ in the presentation of $\Gbar$.
\begin{proof}
	Any group with the relation $x^d = 1$ also has the relations $y^kx^dy^{-k} = 1$ as well as $[y^kx^dy^{-k}, x] = 1$, and since
	$d$ divides $a$ and $b$, that group will also have the relation $y^{-1}x^{a}y = x^b$. 
	Therefore, $\Gbar/A$ has presentation $\langle x, y \mid x^d \rangle$.
\end{proof}
Since we know that $A$ is abelian, we will write $A$ additively instead of multiplicatively. 
Let $a = ud$, $b = vd$. (So $\gcd(u,v) = 1$.)
\begin{proposition}[Moldavanskii]
	The group $A$ has the following presentation as an abelian group (using additive notation)
	\[
	A = \langle c_i, i \in \integers \mid uc_i = vc_{i+1} \text{ for all } i \in \integers \rangle.
	\]
\end{proposition}
Moldavanskii also shows in Proposition 4 of \cite{Moldavanskii} that $A$ is
a residually finite abelian group of rank 1. (For $A$ to have rank 1 means that all of its finitely generated
subgroups are cyclic.) We will show in Lemma \ref{lem:Z[u/v, v/u] iso_to A - Baumslag-Solitar} something similar, that $A$
is isomorphic to $\integers[u/v, v/u] = \{a_1(u/v)^{t_1} + \cdots + a_k(u/v)^{t_k} : a_i, t_i \in \integers \text{ } \forall i \}$.
We remind the reader that 
the ring $\integers[u/v, v/u]$ is $\integers$ together with the two rational numbers $u/v$ and $v/u$ adjoined.
See Lemma \ref{lem:Z[1/pi(uv)] iso_to Z[1/uv] iso_to Z[u/v, v/u]} below for a well-known alternative perspective.

We let $\pi(uv)$ denote the product of the distinct primes that divide $uv$.
\begin{lemma}
	\label{lem:Z[1/pi(uv)] iso_to Z[1/uv] iso_to Z[u/v, v/u]}
	Assume still that $\gcd(u,v) = 1$. As subrings of $\rationals$, we have
	$$\integers[v/u, u/v] = \integers[1/u, 1/v] = \integers[1/(uv)] = \integers[1/\pi(uv)]$$
\end{lemma}
Lemma \ref{lem:Z[1/pi(uv)] iso_to Z[1/uv] iso_to Z[u/v, v/u]} is well-known.

\begin{lemma}
	\label{lem:Z[u/v, v/u] iso_to A - Baumslag-Solitar}
	We have that $A \cong \integers[u/v, v/u]$ as groups.
\end{lemma}
\begin{proof}
	Let $\varphi : \{ c_i : i \in \integers \} \to \integers[u/v, v/u]$ be defined by $\varphi(c_i) := (u/v)^i$. 
	
	\emph{Step 1. $\varphi$ gives a homomorphism:} To get
	a homomorphism from $A$ to $\integers[u/v, v/u]$, all we need to check is that
	$u\varphi(c_i) = v\varphi(c_{i+1})$. And indeed, it is true that $u(u/v)^i = v(u/v)^{i+1}$.
	
	\emph{Step 2. $\varphi$ is surjective:}  This is evident because for all $i$, $(u/v)^i$ is in the image of $\varphi$.
	%
	
	\emph{Step 3. $\varphi$ is injective:} Let $g \in \ker(\varphi)$. Assume by contradiction that $g \neq 0$. Then 
	there exist $n_i \in \integers$ such that $g = \sum_{i=s}^t n_i c_i$ with $n_s, n_t \neq 0$. We will show that
	we can assume that the sum has only one term in it (i.e.\ that $s = t$) and then easily get a contradiction.
	
	We have 
	$\varphi(g) = \sum_{i=s}^t n_i (u/v)^i = 0$. Assume $t > s$. Multiplying by $v^t$ and dividing by $u^s$ yields
	\[
	n_s v^{t-s} + n_{s+1} v^{t-s-1} u + n_{s+2} v^{t-s-2} u^2 + \cdots + n_t u^{t-s} = 0.
	\]
	Therefore $u \mid n_s v^{t-s}$, and since $\gcd(u, v^{t-s}) = 1$, we get $u \mid n_s$. Thus we can rewrite $g$ and then 
	apply the relation $u c_i = v c_{i+1}$ to get 
	\[
	g = \frac{n_s}{u} uc_i + \sum_{i = s+1}^t n_i c_i = \frac{n_s}{u} v c_{i+1} + \sum_{i = s+1}^t n_i c_i.
	\]
	Since we assumed $t > s$, we showed that we can rewrite $g$ as
	$\sum_{i=s+1}^t \tilde{n_i} c_i$, decreasing the number of terms in the summation (by at least 1). Continuing in this way,
	we see that $g = n c_t$ for some $n \in \integers$. Because we assumed $g \neq 0$, we know that $n \neq 0$.
	Therefore $0 = \varphi(g) = \varphi(n c_t) = n(u/v)^t$, and this is a contradiction since $n \neq 0$.
\end{proof}
Recall that $d = \gcd(a,b)$, and $a = ud$, $b = vd$.
\begin{corollary}
	\label{cor:Gbar - short exact sequence - with stated action}
	The group $\Gbar$ (defined after Theorem \ref{thm:Moldavanskii on Res(Baumslag-Solitar)}) 
	satisfies a short exact sequence of the form
	\[
	\integers[1/(uv)] \hookrightarrow \Gbar \twoheadrightarrow \integers * (\integers/d\integers).
	\]
	Writing $\integers * \integers/d\integers = \langle x, y \mid x^d \rangle$,  the action of $x$ on $\integers[1/(uv)]$ is trivial, and the action
	of $y$ on $\integers[1/(uv)]$ is multiplication by $u/v$.
\end{corollary}
\begin{proof}
	Indeed, this is just a summary of the previous results: By Lemma \ref{lem:Gab has A as an abelian normal subgroup},
	$A \normal G$. By Lemma \ref{lem:Z[u/v, v/u] iso_to A - Baumslag-Solitar}, $A \cong \integers[u/v,v/u]$, which
	is isomorphic to $\integers[1/(uv)]$ by Lemma \ref{lem:Z[1/pi(uv)] iso_to Z[1/uv] iso_to Z[u/v, v/u]}. 
	Finally, Corollary \ref{cor:Gbar/A is a free product} gives us the rest of the short exact sequence.
	
	We know that $x$ acts trivially on $\integers[1/(uv)]$ (by conjugation) because in $\Gbar$, the element $x$ commutes with $x^d$, and 
	$x^d$ normally generates $A = \integers[1/(uv)]$.
	
	Finally, consider the relation $y^{-1}x^ay = x^b$ in $\Gbar$. Recall that $a = ud$ and $b  = vd$. So solving the relation for $x^a$, we can rewrite it
	as $(x^d)^u =  y(x^d)^vy^{-1}$. Written additively, this says that $y$ acts on $x^d$ (a generator of $A$) by multiplication by $u/v$.
\end{proof}

\section{When $\gcd(a,b) = 1$: redoing Gelman's formula}
\label{sec:redoing Gelman's formula}
In this section, we give a new proof of a beautiful result of Gelman (Theorem~\ref{thm:Gelman} below).  Although this argument
is a standard derivations argument, one reason to include it is that in the author's 
opinion, this proof better explains the result.
Gelman's formula makes sense in light of the free product $\integers * (\integers/\gcd(a,b)\integers)$ simplifying to $\integers$ and so
giving the semidirect product in Lemma~\ref{lem:gcd(a,b) = 1 - Gbar is a semidirect product}.

As before, we let $\Gab{a}{b} := \langle x, y \mid y^{-1}x^ay = x^b \rangle$. Assume $\gcd(a,b) = 1$.
In \cite{Gelman}, Gelman gives the following exact formula for 
$a_n(\Gab{a}{b})$, the number of
\emph{all} subgroups of index $n$ in  $\Gab{a}{b}$:
\begin{theorem}[Gelman, 2005]
	\label{thm:Gelman}
	Recall that $\gcd(a,b) = 1$. We have
	\[
	a_n(\Gab{a}{b}) = \sum_{\substack{d|n \\ \gcd(d,ab) = 1}} d
	\]
\end{theorem}
In order to (re)prove this, we state a few lemmas. First, we state the isomorphism type of $\Gbar$, the largest residually finite quotient of $\Gab{a}{b}$.
\begin{lemma}
	\label{lem:gcd(a,b) = 1 - Gbar is a semidirect product}
	Let $\Gbar$ be the group defined just after Theorem \ref{thm:Moldavanskii on Res(Baumslag-Solitar)}.
	Then 
	\[\Gbar\cong \Z[1/(ab)] \rtimes \Z,\] 
	where the action of $1 \in \Z$ on 
	$\Z[1/(ab)]$ is multiplication by $a/b$.
\end{lemma}
\begin{proof}
	By Corollary \ref{cor:Gbar - short exact sequence - with stated action}, (and since $d = \gcd(a,b) = 1$), we know that 
	\[
	\integers[1/(ab)] \hookrightarrow \Gbar \twoheadrightarrow  \integers.
	\]
	Because $\Z$ is a free group, every such short exact sequence splits.
	
	The statement about the action also follows from Corollary~\ref{cor:Gbar - short exact sequence - with stated action}: Indeed,
	recall that since $d = \gcd(a,b) = 1$, we have in the notation of that corollary, $u=a$ and $v=b$.
\end{proof}

Once we have Lemma~\ref{lem:gcd(a,b) = 1 - Gbar is a semidirect product}, proving Theorem~\ref{thm:Gelman} is standard. Notice that the group $\Gbar$ 
is an example of a group included in Lemma~3.4, part (i) in \cite{Shalev-On_the_degree}, and Shalev has the formula 
(i.e.\ the one in Theorem~\ref{thm:Gelman})
there in his remark (on page 3804) following his proof of his Lemma~3.4. Nevertheless, we will give a few more details anyways.

Lemma~\ref{lem:Z[1/k] has at most 1 subgroup of any index} is well-known. (We will use it in the following section as well.)
\begin{lemma}
	\label{lem:Z[1/k] has at most 1 subgroup of any index}
	Let $0 \neq k \in \Z$. We have
	\[
	\subgr(\Z[1/k]) =
	\begin{cases}
	1 & \text{if } \gcd(n,k) = 1\\
	0 & \text{otherwise.}
	\end{cases}
	\]
	Also, the nonzero ideals of $\integers[1/k]$ are exactly the subgroups of finite index.
\end{lemma}
%
%

\begin{lemma}[quoted from Shalev]
	\label{lem:derivations - a lemma from Shalev's paper}
	Suppose $A$ is an abelian group, and let $G = A \rtimes B$. Then
	\[
	\subgr(G) = \sum_{A_0, B_0} \lvert \Der(B_0, A/A_0) \rvert,
	\]
	where the sum is taken over all subgroups $A_0 \leq A$, $B_0 \leq B$ such that $A_0$ is $B_0$-invariant, and 
	$[A : A_0] [B : B_0] = n$.
\end{lemma}
This is Lemma 2.1 part (iii) in \cite{Shalev-On_the_degree}.

\begin{proof}[Proof of Theorem~\ref{thm:Gelman}]
	In the notation of Lemma \ref{lem:derivations - a lemma from Shalev's paper}, let $A = \integers[1/(ab)]$ and $B = \integers$, so that as in 
	Lemma~\ref{lem:gcd(a,b) = 1 - Gbar is a semidirect product}, we have $\Gbar \cong A \rtimes B$.
	
	Let $B_0 \leq_f B$ (i.e.\ let $B_0$ be a subgroup of finite index in $B$). Then a subgroup $A_0 \leq_f A$ is $B_0$-invariant iff it is $B$-invariant iff $A_0$ is an ideal of $A$. Recall that since $\integers$
	is a free group, regardless of its action on $\integers/d\integers$, we get that $ \lvert \Der(\integers, \integers/d\integers) \rvert = d$. Combining
	the previous two sentences with
	Lemmas~\ref{lem:derivations - a lemma from Shalev's paper} and \ref{lem:Z[1/k] has at most 1 subgroup of any index}, we conclude that 
	\[\tag{*}
	\subgr(\Gbar) = \sum_{d \mid n} \subgr[n/d](\integers) \subgr[d](\integers[1/(ab)]) d.
	\]
	But $\subgr[n/d](\integers) = 1$, and then using Lemma~\ref{lem:Z[1/k] has at most 1 subgroup of any index} again, (*) becomes
	\[
	\subgr(\Gbar) = \sum_{\substack{d|n \\ \gcd(d,ab) = 1}} d.
	\]
	We are done because $\Gbar$ is the largest residually finite quotient of $\Gab{a}{b}$.
\end{proof}

Gelman's formula simplifies to the following when counting maximal subgroups:

\begin{corollary}
	\label{cor:maxsubgr(Gab) when gcd(a,b) = 1}
	Recall that here, $\gcd(a,b) = 1$. Every maximal subgroup of $\Gab{a}{b}$ has prime index, and 
	\[
	\maxsubgr[p](\Gab{a}{b}) = \begin{cases}
	p + 1 & \text{if } p \nmid ab\\
	0 & \text{otherwise}.
	\end{cases}
	\]
\end{corollary}
\begin{proof}
	The reason why $\Gab{a}{b}$ has no maximal subgroups of non-prime index is because $M \leq G$ with $M$ maximal of index $n$ implies that
	$M \cap \integers[1/(ab)]$ is a maximal ideal of $\integers[1/(ab)]$ of index $n$, and such an $n$ can only be prime. The present corollary
	then follows from Theorem~\ref{thm:Gelman}.
\end{proof}

\section{When $\gcd(a,b) \neq 1$: an asymptotic formula for $\maxsubgr(\Gab{a}{b})$}
\label{sec:when gcd(a,b) > 1, a formula for maxsubgr(Gab)}
In this section, we will write $m := \gcd(a,b)$ and assume $m > 1$.

The main goal of this section is to prove Theorem~\ref{thm:baumslag-solitar - maximal subgroup growth}. 
To this end, the first goal is to prove Proposition \ref{prop:few fixed points for a random permutation of order m}. It says
that (for a fixed number $m$) if $n$ is large, then a random element of $\Sym(n)$ of order
dividing $m$ most likely has very few fixed points.\footnote{A fixed point of a permutation $\sigma \in \Sym(n)$ is an $i \in [n]$ 
	such that $\sigma(i) = i$.} By a result of \cite{Babai-Hayes}, this will imply that a random element of
order $m$ in $\Sym(n)$ (using the uniform distribution) together with a random element of any order, will with high
probability generate
either $A_n$ or $\Sym(n)$. This is what we need in order to calculate the maximal subgroup growth 
of $\integers * \integers/m \integers$
and hence of the Baumslag-Solitar groups (where $\gcd(a,b)  = m \neq 1$).

We first state formula (8) on page 115 of \cite{Muller}.
\begin{theorem}[M\"uller, 1996]
	\label{thm:Muller's asymptotic formula - all finite groups}
	Let $G$ be a finite group of order $m \geq 2$. Then
	$\lvert \Hom(G,\Sym(n)) \rvert$ is asymptotic to 
	\[
	K_G n^{(1-1/m)n} \exp\left(-(1 - 1/m)n + \sum_{\substack{d < m \\ d|m}} \frac{\subgr[d](G)}{d} n^{d/m}\right),
	\]
	where
	\[
	K_G :=
	\begin{cases}
	m^{-1/2} & \text{if $m$ is odd}\\
	m^{-1/2}e^{-\subgr[m/2](G)^2/(2m)} & \text{otherwise.}
	\end{cases}
	\]
\end{theorem}


We will use the following easy consequence of Theorem \ref{thm:Muller's asymptotic formula - all finite groups}:
\begin{corollary}
	\label{cor:Muller's asymptotic formula - cyclic groups}
	Let  $m \geq 2$. Then
	$\lvert \Hom(\integers/m\integers,\Sym(n)) \rvert$ is asymptotic to $f(n)$, where
	\[\tag{*}
	f(n) := K_m n^{(1-1/m)n} \exp\left(-(1 - 1/m)n + \sum_{\substack{d < m \\ d|m}} \frac{n^{d/m}}{d} \right),
	\]
	and
	\[
	K_m :=
	\begin{cases}
	m^{-1/2} & \text{if $m$ is odd}\\
	m^{-1/2}e^{-1/(2m)} & \text{otherwise.}
	\end{cases}
	\]
\end{corollary}

\begin{proposition}
	\label{prop:few fixed points for a random permutation of order m}
	Fix an integer $m \geq 2$. Let $M(n)$ denote the elements of $\Sym(n)$ of order dividing $m$,
	and let $B(n)$ denote the elements of $M(n)$ which have at least $g(n) := \lfloor n/ \log(n) \rfloor$ fixed points.
	Then 
	\[
	\frac{|B(n)|}{|M(n)|} \longrightarrow 0.
	\]
\end{proposition}
\begin{proof}[Proof of reduction to Lemma \ref{lem:nChoose_k times fraction approaches 0}]
	Write the group $\Sym(n)$ as $\Sym([n])$, and so for any $\Omega \subseteq [n]$ 
	(with $|\Omega| \geq 1$) we can consider the subgroup $\Sym(\Omega)$ of 
	$\Sym([n])$.\footnote{$\Sym(\Omega)$ is the group $\{\sigma \in \Sym([n]) \mid \sigma(i) = i 
		\text{ for all } i \in [n] \setminus  \Omega \}$. }
	
	Notice that we have the following equality of sets:
	\[
	B(n) = \bigcup_{\substack{\Omega \subseteq [n] \\ |\Omega| = n - g(n)}} B(n) \cap \Sym(\Omega).
	\]
	Also, note that $|B(n) \cap \Sym(\Omega)| \leq |M(|\Omega|)| = |M(n - g(n))|$. Therefore,
	\[\begin{aligned}
	|B(n)| &\leq \sum_{\substack{\Omega \subseteq [n] \\ |\Omega| = n - g(n)}} |M(n - g(n))| \\
	&= \binom{n}{n - g(n)} |M(n - g(n))| = \binom{n}{g(n)} |M(n - g(n))| .
	\end{aligned}\]
	Therefore, to prove this proposition, it is sufficient to show
	\[\tag{**}
	\binom{n}{g(n)} \frac{|M(n-g(n))|}{|M(n)|} \longrightarrow 0.
	\]
	
	Notice that once we choose a generator $x$ of $C_m$, the cyclic group of order $m$, there is an obvious 
	bijection\footnote{A homomorphism from $C_m$ to any group $G$ 
		is just a choice of where to send $x$: any element of $G$
		of order dividing $m$.} 
	between $M(n)$ and $\Hom(C_m, \Sym(n))$. 
	Thus Corollary \ref{cor:Muller's asymptotic formula - cyclic groups} says that $|M(n)|$ is asymptotic to
	the function $f(n)$ (defined in that corollary). Of course, $M(n - g(n))$ is thus asymptotic to 
	$f(n - g(n))$. Combining these observations with (**), we conclude that all we need to show is the following
	(which is Lemma \ref{lem:nChoose_k times fraction approaches 0}):
	\[
	\binom{n}{g(n)} \frac{f(n - g(n))}{f(n)} \longrightarrow 0.
	\]
\end{proof}

Before proving Lemma \ref{lem:nChoose_k times fraction approaches 0}, we first obtain a preliminary lemma.
\begin{lemma}
	\label{lem:upper bound on nChoose_k}
	Let $\displaystyle{g(n):= \left \lfloor{\frac{n}{\log n}}\right \rfloor}$.  For all large $n$, 
	$\displaystyle{ \binom{n}{g(n)}} < e^{3 g(n) \log \log (n)}$.
\end{lemma}
\begin{proof} 
	We have $\displaystyle{ \binom{n}{g(n)} < \frac{n^{g(n)}}{g(n)!}}$. Using Stirling's formula (a lower bound) on the
	denominator, we get
	\[
	\frac{n^{g(n)}}{g(n)!} < \frac{n^{g(n)}e^{g(n)}}{\sqrt{2\pi g(n)} g(n)^{g(n)}}
	\]
	We have that
	\[
	\frac{n^{g(n)}e^{g(n)}}{\sqrt{2\pi g(n)} g(n)^{g(n)}} < \frac{n^{g(n)}}{ g(n)^{g(n)}} e^{g(n)} \\
	< \frac{n^{g(n)}}{ g(n)^{g(n)}} e^{g(n) \log \log n}.
	\]
	So it is sufficient to show that
	\[
	\frac{n^{g(n)}}{g(n)^{g(n)}} < e^{2g(n) \log \log n}. 
	\]
	We have 
	\[
	\frac{n^{g(n)}}{g(n)^{g(n)}}  = \frac{n^{g(n)}}{\left\lfloor \frac{n}{\log n} \right\rfloor^{g(n)}}  
	< \frac{n^{g(n)}}{\left(\frac{n}{\log n} - 1\right)^{g(n)}}
	= \frac{n^{g(n)}}{\left(\frac{n - \log n}{\log n}\right)^{g(n)}},
	\]
	and also
	\[
	\left(\frac{n}{n - \log(n)} \right)^{g(n)} (\log n)^{g(n)} < (\log n)^{2 g(n)} = e^{2g(n) \log \log n}.
	\]
\end{proof}

\begin{lemma}
	\label{lem:nChoose_k times fraction approaches 0}
	Let $f(n)$ be as in Corollary \ref{cor:Muller's asymptotic formula - cyclic groups} and 
	$g(n) := \left \lfloor{\frac{n}{\log n}}\right \rfloor$. Then
	\[
	\binom{n}{g(n)} \frac{f(n - g(n))}{f(n)} \longrightarrow 0.
	\]
\end{lemma}
\begin{proof}
	By Lemma \ref{lem:upper bound on nChoose_k}, it is sufficient to show 
	\[
	\frac{e^{3 g(n) \log \log (n)}f(n - g(n))}{f(n)} \longrightarrow 0.
	\]
	
	Let $h(n) := \sum_{\substack{d < m \\ d|m}} \frac{n^{d/m}}{d}$ and $k = 1-1/m$. We have that 
	\[
	f(n) = \frac{K_m n^{nk} e^{h(n)}}{e^{nk}}.
	\]
	Therefore,
	\[
	\begin{aligned}
	\frac{e^{3 g(n) \log \log n}f(n - g(n))}{f(n)} 
	&= \frac{e^{3g(n) \log \log n} (n-g(n))^{(n-g(n))k} e^{h(n-g(n))}e^{nk}}{n^{nk}e^{(n-g(n))k}e^{h(n)}}\\
	&= \frac{(n-g(n))^{nk}}{n^{nk}} \cdot \frac{e^{h(n-g(n))}}{e^{h(n)}} \cdot \frac{e^{3g(n) \log \log n + g(n)k}}{(n-g(n))^{g(n)k}}\\
	&< \frac{e^{3g(n) \log \log n + g(n)k}}{(n-g(n))^{g(n)k}}\\
	\end{aligned}
	\]
	For $n > e^e$ we have $\log \log n > 1 > 1 - 1/m = k$ 
	(and also $n - g(n) > n/e$, since recall that $g(n) =  \left \lfloor n/\log n \right \rfloor $). Therefore for such
	$n$ we have
	\[
	\begin{aligned}
	\frac{e^{3g(n) \log \log n + g(n)k}}{(n-g(n))^{g(n)k}}
	&< \frac{e^{4g(n)\log \log n}}{(n/e)^{g(n)k}}\\
	&= \frac{e^{4g(n)\log \log n + g(n)k}}{(e^{\log n})^{g(n)k}}\\
	&< \frac{e^{5g(n)\log \log n}}{e^{kg(n)\log n}}\\
	&= \frac{1}{e^{(k\log n - 5\log \log n)g(n)}},
	\end{aligned}
	\]
	which approaches 0 as $n \to \infty$.
\end{proof}

This completes the proof of
Proposition \ref{prop:few fixed points for a random permutation of order m}.

The following is the main theorem in \cite{Babai-Hayes}.  We use ``little $o$'' notation. Also, 
for a permutation group $G \leq \Sym(n)$,
\[\Fix(G) := \text{the number of fixed points of $G$}\] 
\begin{theorem}[Babai \& Hayes, 2006]
	\label{thm:Babai-Hayes}
	Suppose for all large $n$, we have subgroups $G_n \leq \Sym(n)$ with $\Fix(G_n) = o(n)$. 
	Let $\sigma_n \in \Sym(n)$ each be
	chosen at random (with uniform distribution). Then 
	\[\Pr(\langle G_n, \sigma_n \rangle \in \{\Alt(n), \Sym(n) \}) \longrightarrow 1.\]
\end{theorem}

\begin{corollary}
	\label{cor:Babai-Hayes}
	For all large $n$, let $E_n$ be the event that an element $g_n \in \Sym(n)$ has
	$\Fix(\langle g_n \rangle) \leq \lfloor n/ \log n \rfloor$. Then choosing $\sigma_n \in \Sym(n)$ at random,
	we have 
	\[
	\Pr(\langle g_n, \sigma_n \rangle = \text{a primitive subgroup of } \Sym(n), \text{ given } E_n) \longrightarrow 1
	\]
\end{corollary}

We will use the following notation: 
\[\trans(G) :=  \text{the number of transitive permutation representations of $G$ of degree $n$}\]
\[\prim(G) :=  \text{the number of primitive permutation representations of $G$ of degree $n$}\]
\begin{lemma}
	\label{lem:subgroup growth via permutation representations}
	With the above notation, for all $n$ we have
	\[\begin{aligned}
	\subgr(G)       &= \trans(G)/(n-1)! \\
	\maxsubgr(G) &= \prim(G)/(n-1)!
	\end{aligned}
	\]
\end{lemma}
For a proof, see Proposition 1.1.1 on page 12 of \cite{Lubotzky-and-Segal}.

\begin{theorem}
	\label{thm:free product - Z*Z/mZ - maximal subgroup growth}
	Fix $m \neq 1$. Let $G = \integers * \integers/ m\integers$. Let $f$ be as in 
	Corollary \ref{cor:Muller's asymptotic formula - cyclic groups}. Then
	\[
	\maxsubgr(G) \sim n f(n).
	\]
\end{theorem}
\begin{proof}
	By Lemma \ref{lem:subgroup growth via permutation representations}, we just need to show
	\[
	\prim(G) \sim n! \cdot f(n).
	\]
	Let $\hom(G) := \lvert \Hom(G, \Sym(n)) \rvert$. We know that 
	\[
	\hom(G) = n! \cdot \lvert \Hom(\integers/m\integers, \Sym(n)) \rvert,
	\]
	and Corollary \ref{cor:Muller's asymptotic formula - cyclic groups} says that  
	$\lvert \Hom(\integers/m\integers, \Sym(n))\rvert \sim f(n)$. So all we need to show is that
	\[\tag{*}
	\prim(G) \sim \hom(G).
	\]
	
	As we shall see, (*) follows from Corollary \ref{cor:Babai-Hayes}, which we can apply
	because of Proposition \ref{prop:few fixed points for a random permutation of order m}. The rest of this proof
	just fills in the details, explaining the previous sentence.
	
	We will use the notation $M(n)$ from Proposition \ref{prop:few fixed points for a random permutation of order m}. 
	Also, let $E_n$ be the event that a random element, $g$, of $M(n)$ has fewer than $\lfloor n/ \log(n) \rfloor$ fixed points. 
	Proposition \ref{prop:few fixed points for a random permutation of order m} (slightly reworded) says that
	\[
	\Pr(E_n) \longrightarrow 1.
	\]
	
	Next, in addition to choosing an element $g$ of $\Sym(n)$ of order dividing $m$, we also will choose a random
	element $\sigma \in \Sym(n)$. Hence, we are just choosing a random homomorphism from $G$ to $\Sym(n)$. 
	Let $W_n$ be the event that $\langle g, \sigma \rangle$ is a primitive subgroup of $\Sym(n)$.
	Because $\Pr(W_n) = \prim(G)/\hom(G)$, notice that in order to prove (*), we need to show that
	\[
	\Pr(W_n) \longrightarrow 1.
	\]
	
	By Corollary \ref{cor:Babai-Hayes}, we know that 
	$\Pr(W_n \text{ given } E_n) \longrightarrow 1$, and Proposition~\ref{prop:few fixed points for a random permutation of order m}
	tells us that $\Pr(E_n) \longrightarrow 1$. Therefore,
	\[
	\Pr(W_n) \geq \Pr(W_n \text{ and } E_n) = \Pr(W_n \text{ given } E_n) \cdot \Pr(E_n) \longrightarrow 1.
	\]
\end{proof}

We almost have Theorem \ref{thm:baumslag-solitar - maximal subgroup growth}, which 
gives the
maximal subgroup growth of the Baumslag-Solitar groups. The only thing we need to do first is show that these groups have
very few maximal subgroups that are not contained in the quotient   $\integers * \integers/ m\integers$. In other words,
our goal is to show that
Theorem \ref{thm:free product - Z*Z/mZ - maximal subgroup growth} is sufficient to count almost all of the
maximal subgroups. To achieve this goal (in Lemma~\ref{lemma:baumslag-solitar - polynomial bound on the maximal subgroups of complent type}), 
we state another lemma first.

\begin{lemma} 
	\label{lem:abelian by something - counting maximal subgroups by derivations}
	Let $G$ be a f.g.\ group with $A \normal G$ and $A$ abelian. Then
	\[ \tag{*}
	\maxsubgr(G) \leq \maxsubgr(G/A) + \sum_{A_0} \lvert\Der(G/A, A/A_0)\rvert
	\]
	where the sum is taken over all $A_0$ such that $A_0 \normal G$, $A_0 \leq A$ and such that $A/A_0$ is a simple $\integers[G/A]$-module with $|A/A_0| = n$. 
	When we have $G \cong A \rtimes G/A$, then the inequality in (*) is an equality.
\end{lemma}
Lemma \ref{lem:abelian by something - counting maximal subgroups by derivations} is Lemma 5 in \cite{Kelley_metabelian_groups}.

\begin{lemma}
	\label{lemma:baumslag-solitar - polynomial bound on the maximal subgroups of complent type}
	Let $\Gbar$ and $A$ be as in Section \ref{sec:largest residually finite quotient}.
	Let $\mc = \maxsubgr(\Gbar) - \maxsubgr(\Gbar/A)$. Then 
	\[
	\begin{aligned}
	\mc &= 0         \text{\quad if $n$ is not prime, and}\\
	\mc[p] &\leq p^2 \text{\quad if $p$ is prime}. \\
	\end{aligned}
	\]
\end{lemma}
\begin{proof}
	Because $\Z[1/(uv)]$ has no maximal submodules that are not of prime index, 
	Lemma~\ref{lem:abelian by something - counting maximal subgroups by derivations}
	implies that $\mc = 0$ for such $n$. 
	
	Let $n =p$ be prime. We know that $\Z[1/(uv)]$ has at most 1 maximal ideal of index $p$ 
	(by, say Lemma \ref{lem:Z[1/k] has at most 1 subgroup of any index}). Therefore, to show that 
	$\mc[p] \leq p^2$, by
	Lemma~\ref{lem:abelian by something - counting maximal subgroups by derivations}, we just need to show that
	$\lvert \Der(\Z * \Z/m\Z, \Z/p\Z)\rvert \leq p^2$. This is immediate because the number of functions
	from a two element generating set of $\Z * \Z/m\Z$ to $\Z/p\Z$ is at most $p^2$.
	
\end{proof}

\begin{theorem}
	\label{thm:baumslag-solitar - maximal subgroup growth}
	Let $m = \gcd(a,b) > 1$. 
	Then
	\[
	\maxsubgr(\Gab{a}{b}) \sim K_m n^{(1-1/m)n + 1} \exp\left(-(1 - 1/m)n + \sum_{\substack{d < m \\ d|m}} \frac{n^{d/m}}{d} \right),
	\]
	where
	\[
	K_m :=
	\begin{cases}
	m^{-1/2} & \text{if $m$ is odd}\\
	m^{-1/2}e^{-1/(2m)} & \text{otherwise.}
	\end{cases}
	\]
\end{theorem}
\begin{proof}
	Let $f$ be as in Corollary~\ref{cor:Muller's asymptotic formula - cyclic groups}, $\Gbar$ from immediately after
	Theorem~\ref{thm:Moldavanskii on Res(Baumslag-Solitar)}, $A$ from 
	Lemma~\ref{lem:Gab has A as an abelian normal subgroup}, and
	$\mc$ as in Lemma~\ref{lemma:baumslag-solitar - polynomial bound on the maximal subgroups of complent type}.
	
	We know that $\maxsubgr(\integers * \integers/m\integers) \leq \maxsubgr(\Gab{a}{b})$, because 
	Corollary~\ref{cor:Gbar - short exact sequence - with stated action} tells us that $\Gbar/A \cong \integers * \integers/m\integers$. 
	So by Theorem~\ref{thm:free product - Z*Z/mZ - maximal subgroup growth}, we observe that $\maxsubgr(\Gab{a}{b})  = \maxsubgr(\Gbar)$ grows at least as fast as
	$nf(n)$.
	
	By definition of $\mc$ (in Lemma~\ref{lemma:baumslag-solitar - polynomial bound on the maximal subgroups of complent type}) 
	we can write 
	\[
	\maxsubgr(\Gbar) = \maxsubgr(\Gbar/A) + \mc
	\]
	
	We are done because Lemma~\ref{lemma:baumslag-solitar - polynomial bound on the maximal subgroups of complent type}
	gives us that $\mc$ is bounded above by a polynomial of degree 2.	
\end{proof}

\section{When $\gcd(a,b) \neq 1$: an asymptotic formula for $\subgr(\Gab{a}{b})$}
\label{sec:when gcd(a,b) > 1, formula for subgr(Gab)}
The goal of this section is Theorem~\ref{thm:baumslag-solitar - subgroup growth}. 
In this section, we will again denote $\gcd(a,b)$ by $m$, and we assume $m > 1$. The following result follows from the proof of
Theorem~\ref{thm:free product - Z*Z/mZ - maximal subgroup growth}.

\begin{corollary}
	\label{cor:subgr(Z * Z/mZ) is asymptotic to nf(n)}
	Let $G = \integers * \integers/m\integers$. Let $f$ be as in Corollary~\ref{cor:Muller's asymptotic formula - cyclic groups}. Then
	\[
	\subgr(G) \sim nf(n). 
	\]
\end{corollary}
\begin{proof}
	By Lemma \ref{lem:subgroup growth via permutation representations}, we just need to show
	\[
	\trans(G) \sim n! \cdot f(n).
	\]
	Let $\hom(G) := \lvert \Hom(G, Sym(n)) \rvert$. We know that 
	\[
	\hom(G) = n! \cdot \lvert \Hom(\integers/m\integers, \Sym(n)) \rvert,
	\]
	and Corollary \ref{cor:Muller's asymptotic formula - cyclic groups} says that  
	$\lvert \Hom(\integers/m\integers, \Sym(n))\rvert \sim f(n)$. So all we need to show is that
	\[\tag{*}
	\trans(G) \sim \hom(G).
	\]
	
	In the proof of Theorem~\ref{thm:free product - Z*Z/mZ - maximal subgroup growth}, 
	we showed that $\prim(G) \sim \hom(G)$. We have that $\prim(G) \leq \trans(G) \leq \hom(G)$. Therefore 
	$\trans(G) \sim \hom(G)$.
\end{proof}

Let $A$ be as in Section 2. The purpose of Lemmas~\ref{lem:subgr(G) - upper bound using derivations - Shalev} through
\ref{lem:fraction in terms of g(n) approaches 0} is to show that the vast majority of all subgroups
of $\Gab{a}{b}$ (of any fixed, large index) contain $A$. This implies that the formula in 
Corollary~\ref{cor:subgr(Z * Z/mZ) is asymptotic to nf(n)} also works for $\subgr(\Gab{a}{b})$.

\begin{lemma}
	\label{lem:subgr(G) - upper bound using derivations - Shalev}
	Let $G$ be a group, and let $A \normal G$ with $A$ abelian. Then
	\[
	  \subgr(G) \leq \sum_{d \divides n} \subgr[n/d](G/A) \subgr[d](A) D_{n,d},
	\]
	where $D_{n,d} = \max_{A_0,G_0} \lvert \Der(G_0/A, A/A_0) \rvert$, where the $\max$ is over the subgroups
	$A_0 \leq A \leq G_0 \leq G$ with $[A : A_0] = d$, $[G : G_0] = n/d$, and $A_0 \normal G_0$.
\end{lemma}
This is part of Lemma 2.1 part (ii) in \cite{Shalev-On_the_degree}.

In what follows, $a = um$ and $b = vm$. 
\begin{lemma}
	\label{lem:der(G_0/A,A/A_0) is bounded above by 3^(2n/3)}
  Let $\Gbar$ be the group defined just after Theorem~\ref{thm:Moldavanskii on Res(Baumslag-Solitar)}. Let $A \cong \integers[1/(uv)]$ be the subgroup
  of $\Gbar$ in Corollary~\ref{cor:Gbar - short exact sequence - with stated action} so that $\Gbar/A \cong \integers * \integers/m\integers$.
  Fix $n > 1$, and let $d \divides n$. Let $G_0 \normal \Gbar$ with $[\Gbar : G_0] = n/d$, and let $A_0 \leq_d A$. Then
  \[
  \lvert \Der(G_0/A,A/A_0) \rvert \leq 3^{2n/3}.
  \]
\end{lemma}
Note that this result basically follows from the proof of Proposition 1.3.2 part (i) in \cite{Lubotzky-and-Segal}.
\begin{proof}
	Recall that for a f.g.\ group $H$, we let $d(H)$ denote the minimal size of a generating set for $H$. Hopefully this notation will not be confusing because
	$n/d$ is the index of $G_0$ in $\Gbar$.
	
	We have that $2 = d(\Gbar/A)$. By Schreier's formula (Result 6.1.1 in \cite{Robinson}), we have that
	\[
	d(G_0/A) \leq 1 + [\Gbar : G_0](2 -1) = 1 + \frac{n}{d} \leq \frac{2n}{d}.
	\]
	Therefore,
	\[
	  \lvert \Der(G_0/A,A/A_0) \rvert \leq |A/A_0|^{d(G_0/A)} \leq d^{2n/d} \leq 3^{2n/3},
	\]
	since $d^{1/d} \leq 3^{1/3}$ for every $d \in \mathbb{N}$.
\end{proof}

\begin{lemma}
	\label{lem:bounding subgr(G) above and below by g(n)}
  Let $G = \Gab{a}{b}$ (with $m = \gcd(a,b) > 1$). Let $f$ be as in Corollary~\ref{cor:Muller's asymptotic formula - cyclic groups}, and let 
  $g(n) = nf(n)$. Let $\varepsilon > 0$. Then for all large $n$,
  \[
   (1 - \varepsilon)g(n) \leq \subgr(G) \leq (1 + \varepsilon)g(n) + (1 + \varepsilon)ng(n/2) 3^{2n/3}.
  \]
\end{lemma}
\begin{proof}
	Let $\Gbar$ be the group defined just after Theorem~\ref{thm:Moldavanskii on Res(Baumslag-Solitar)}. So $\subgr(\Gbar) = \subgr(G)$.
	 Let $A \cong \integers[1/(uv)]$ be the subgroup
	of $\Gbar$ in Corollary~\ref{cor:Gbar - short exact sequence - with stated action} so that $\Gbar/A \cong \integers * \integers/m\integers$. 
	 Suppose $n$ is large. We have that 
	\[
	  \subgr(\integers * \integers/m\integers) \leq \subgr(\Gbar).
	\]
	
	By Corollary \ref{cor:subgr(Z * Z/mZ) is asymptotic to nf(n)}, $\subgr(\integers * \integers/m\integers) \sim g(n)$. Therefore, since $n$ is large,
	\[
	  (1 - \varepsilon)g(n) \leq \subgr(\integers * \integers/m\integers) \leq \subgr(\Gbar),
	\]
	which proves the lower bound in this lemma.
	
	Next, by Lemma~\ref{lem:subgr(G) - upper bound using derivations - Shalev}, we have that
	\[\tag{*}
	  \subgr(\Gbar) \leq \sum_{d \divides n} \subgr[n/d](\integers * \integers/m\integers)\subgr[d](A) D_{n,d}.
	\]
	We have by Lemma~\ref{lem:Z[1/k] has at most 1 subgroup of any index} that $\subgr[d](A) \leq 1$ for all $d$. Also, $D_{n,1}  = 1$. Further,
	by Lemma~\ref{lem:der(G_0/A,A/A_0) is bounded above by 3^(2n/3)},  $D_{n,d} \leq 3^{2n/3}$. Therefore, from (*) we conclude that
	\[
	  \subgr(\Gbar) \leq \subgr(\integers * \integers/m\integers) + \sum_{\substack{ d \divides n \\ d > 1}} \subgr[n/d](\integers * \integers/m\integers) \cdot 3^{2n/3},
	\]
	which is bounded above by
	\[
	  (1 + \varepsilon)g(n) + \sum_{d = 1}^n (1 + \varepsilon) g(n/2) 3^{2n/3}
	\]
	since $n$ is large, $(\subgr(\integers * \integers/m\integers)) \sim g(n)$, and $g(n)$ is an increasing function. This proves the upper bound of the
	lemma.
\end{proof}

\begin{lemma}
	\label{lem:sum n^(d/m)/d < n}
	Recall $m \geq 2$. For $n \geq m$, we have
	\[
	  \sum_{\substack{d < m\\ d \divides m}} \frac{n^{d/m}}{d} < n.
	\]
\end{lemma}
\begin{proof}
	We denote by $\tau(m)$ the number of divisors of $m$. From elementary number theory (see for example, the top of page 114 in \cite{LeVeque}), 
	$\tau(m) \leq \sqrt{3m}$. 
	Therefore,
	\[\begin{aligned}
	\sum_{\substack{d < m\\ d \divides m}} \frac{n^{d/m}}{d} &\leq \sum_{\substack{d < m\\ d \divides m}} \frac{n^{1/2}}{d}
	< \sum_{\substack{d < m\\ d \divides m}} \frac{n^{1/2}}{2} \leq \tau(m) \frac{n^{1/2}}{2}\\
	&\leq \frac{\sqrt{3m}\sqrt{n}}{2} < \sqrt{m}\sqrt{n} \leq \sqrt{n}\sqrt{n} = n.
	\end{aligned}
	\]
\end{proof}

\begin{lemma}
	\label{lem:bounding g(n) above and below}
	Let $f$ and $K = K_m$ be as in Corollary~\ref{cor:Muller's asymptotic formula - cyclic groups}, and let $g(n) = nf(n)$. Let $\delta = 1 - 1/m$. 
	Then
	\[
	  Kn \left(\frac{n}{e}\right)^{\delta n} < g(n) < K n \left(\frac{n}{e}\right)^{\delta n} e^{n}.
	\]
\end{lemma}
\begin{proof}
	We have that 
	\[
	  g(n) = Kn \left(\frac{n}{e}\right)^{\delta n} \exp(\sum_{\substack{d < m\\ d \divides m}} \frac{n^{d/m}}{d}).
	\]
	So the lower bound on $g(n)$ follows. The upper bound follows from Lemma~\ref{lem:sum n^(d/m)/d < n}.
\end{proof}

\begin{lemma}
	\label{lem:fraction in terms of g(n) approaches 0}
	Let $f$ and $K = K_m$ be as in Corollary~\ref{cor:Muller's asymptotic formula - cyclic groups}, and let $g(n) = nf(n)$. Then
	\[
	  \frac{n3^{2n/3}g(n/2)}{g(n)} \longrightarrow 0.
	\]
\end{lemma}
\begin{proof}
	We will show that
	\[
	\frac{3^n g(n/2)}{g(n)} \longrightarrow 0.
	\]
	Let $f_0(n) = Kn \left(\frac{n}{e}\right)^{\delta n}$ and $f_1(n) = K n \left(\frac{n}{e}\right)^{\delta n} e^{n}$. 
	By Lemma~\ref{lem:bounding g(n) above and below},
	\[
	  \frac{3^ng(n/2)}{g(n)} < \frac{3^nf_1(n/2)}{f_0(n)}.
	\]
	So we will just show that $3^nf_1(n/2)/f_0(n) \longrightarrow 0$. Indeed,
	\[\begin{aligned}
	 \frac{3^nf_1(n/2)}{f_0(n)} &= \frac{3^n K \frac{n}{2} \left(\frac{n}{2e}\right)^{\delta n/2} e^{n/2}}{Kn\left(\frac{n}{e}\right)^{\delta n}}
	= \frac{e^{n\log 3} \left(\frac{1}{2}\right)^{\delta n /2 + 1} \left(\frac{n}{e}\right)^{\delta n/2} e^{n/2}}{\left(\frac{n}{e}\right)^{\delta n}}\\
	&< \frac{e^{n\log3 +n/2}}{\left(\frac{n}{e}\right)^{\delta n/2}} = \frac{e^{n\log3 +n/2 + \delta n/2}}{e^{(\delta /2)n\log n}}
	= \frac{1}{e^{(\delta /2) n \log n - (\log 3 + 1/2 + \delta /2)n}},
	\end{aligned}
	\]
	which goes to 0.
\end{proof}

Given Lemma \ref{lem:bounding subgr(G) above and below by g(n)}, the purpose of 
Lemma~\ref{lem:h(n) bounded in terms of g(n) implies h(n) is asymptotic to g(n)} should be clear.
\begin{lemma}
	\label{lem:h(n) bounded in terms of g(n) implies h(n) is asymptotic to g(n)}
	Let $f$ be as in Corollary~\ref{cor:Muller's asymptotic formula - cyclic groups}, and let $g(n) = nf(n)$. Suppose
	$h : \mathbb{N} \to \mathbb{N}$ is such that for all $\varepsilon > 0$, for all large $n$,
	\[
	  (1 - \varepsilon)g(n) \leq h(n) \leq (1 + \varepsilon)g(n) + (1 + \varepsilon)ng(n/2) 3^{2n/3}.
	\]
	Then $h(n) \sim g(n)$.
\end{lemma}
\begin{proof}
	Fix $\varepsilon > 0$. For all large $n$, we have that
	\[
	1 - \varepsilon = \frac{(1-\varepsilon)g(n)}{g(n)} \leq \frac{h(n)}{g(n)} \leq \frac{(1+\varepsilon)g(n) + (1+\varepsilon)n3^{2n/3}g(n/2)}{g(n)},
	\]
	which by Lemma \ref{lem:fraction in terms of g(n) approaches 0}, approaches $1 + \varepsilon$. So we have shown that for all $\varepsilon > 0$,
	for all large $n$,
	\[
	  1 - \varepsilon \leq \frac{h(n)}{g(n)} \leq 1 + 2\varepsilon.
	\]
	Therefore, $h(n) \sim g(n)$.
\end{proof}

\begin{theorem}
	\label{thm:baumslag-solitar - subgroup growth}
	Let $m = \gcd(a,b) > 1$. 
	Then
	\[
	\subgr(\Gab{a}{b}) \sim K_m n^{(1-1/m)n + 1} \exp\left(-(1 - 1/m)n + \sum_{\substack{d < m \\ d|m}} \frac{n^{d/m}}{d} \right),
	\]
	where
	\[
	K_m :=
	\begin{cases}
	m^{-1/2} & \text{if $m$ is odd}\\
	m^{-1/2}e^{-1/(2m)} & \text{otherwise.}
	\end{cases}
	\]
\end{theorem}
\begin{proof}
	Let $f$ be as in Corollary~\ref{cor:Muller's asymptotic formula - cyclic groups}, and let $g(n) = nf(n)$. The statement of this theorem is that
	$\subgr(\Gab{a}{b}) \sim g(n)$. By Lemma~\ref{lem:bounding subgr(G) above and below by g(n)}, we have that for all $\varepsilon > 0$, for all
	large $n$,
	\[
	  (1 - \varepsilon)g(n) \leq \subgr(G) \leq (1 + \varepsilon)g(n) + (1 + \varepsilon)ng(n/2) 3^{2n/3}.
	\]
	The theorem then follows by Lemma \ref{lem:h(n) bounded in terms of g(n) implies h(n) is asymptotic to g(n)}.
\end{proof}

\begin{corollary}
	Let $m = \gcd(a,b) > 1$. Then $\maxsubgr(\Gab{a}{b}) \sim \subgr(\Gab{a}{b})$.
\end{corollary}
\begin{proof}
	This follows from Theorems \ref{thm:baumslag-solitar - maximal subgroup growth} and \ref{thm:baumslag-solitar - subgroup growth}.
\end{proof}

\bibliography{my_bib}{}

\begin{thebibliography}{10}

\bibitem{Babai-Hayes}
L.~Babai and T.~Hayes.
\newblock {The probability of generating the symmetric group when one of the
  generators is random.}
\newblock {\em Publ. Math. Debrecen}, 69(no. 3):271--280, 2006.

\bibitem{Button}
J.~O. Button.
\newblock {A formula for the normal subgroup growth of Baumslag-Solitar
  groups}.
\newblock {\em J. Group Theory}, 11(6):879--884, 2008.

\bibitem{Gelman}
E.~Gelman.
\newblock {Subgroup growth of Baumslag-Solitar groups}.
\newblock {\em J. Group Theory}, 8(no. 6):801--806, 2005.

\bibitem{kelley-dissertation}
A.~Kelley.
\newblock {\em Maximal Subgroup Growth of Some Groups}.
\newblock PhD thesis, State University of New York at Binghamton, 2017.

\bibitem{Kelley_metabelian_groups}
A.~Kelley.
\newblock Maximal subgroup growth of some metabelian groups.
\newblock {\em preprint arXiv:1807.03423}, 2018.

\bibitem{LeVeque}
W.~LeVeque.
\newblock {\em {Topics in Number Theory, Volume I}}.
\newblock Addison-Wesley, 1958.

\bibitem{Lubotzky-and-Segal}
A.~Lubotzky and D.~Segal.
\newblock {\em {Subgroup growth}}.
\newblock Birkhauser Verlag, Basel, 2003.

\bibitem{Moldavanskii}
D.~I. Moldavanskii.
\newblock {On the intersection of subgroups of finite index in the
  Baumslag-Solitar groups.}
\newblock {\em Math. Notes}, 87(no. 1-2):88--95, 2010.

\bibitem{Muller}
T.~M\"uller.
\newblock {Subgroup growth of free products.}
\newblock {\em Invent. Math.}, 126(no. 1):111--131, 1996.

\bibitem{Robinson}
D.~Robinson.
\newblock {\em {A course in the theory of groups.}}
\newblock Springer-Verlag, New York, second edition, 1996.

\bibitem{Shalev-On_the_degree}
A.~Shalev.
\newblock {On the degree of groups of polynomial subgroup growth.}
\newblock {\em Trans. Amer. Math. Soc.}, 351(no. 9):3793--3822, 1999.

\end{thebibliography}
\bibliographystyle{plain}

\textsc{Department of Mathematics and Computer Science, Colorado College,
	Colorado Springs, Colorado 80903}\par\nopagebreak
\textit{email address}: \texttt{akelley@coloradocollege.edu}\footnote{The author is a Visiting Assistant Professor of Mathematics at Colorado College.
A possibly more permanent email address is \texttt{akelley2500@gmail.com} }

\end{document}